\documentclass[12pt]{article}
\usepackage[textsize=tiny]{todonotes}
\usepackage{caption}
\usepackage{amsmath}                            
\usepackage{amssymb}                            
\usepackage{amsthm}                             
\usepackage{algorithm}
\usepackage[utf8]{inputenc}
\usepackage{enumitem}
\usepackage{centernot}
\usepackage{placeins}
\usepackage{amsthm}
\usepackage{amsmath}
\usepackage{mathrsfs}
\usepackage{amssymb}
\usepackage{scalefnt}
\usepackage{amsfonts}
\usepackage{xcolor}
\usepackage{color}
\usepackage{xspace}
\usepackage{microtype}
\usepackage[mathscr]{eucal}
\usepackage{float}
\usepackage{amsthm}

\DeclareMathOperator{\dcdot}{\cdot\cdot}
\usepackage{centernot}
\usepackage{tikz}
\usepackage{placeins}
\usepackage{amsthm}
\usepackage{amsmath}
\usepackage{amssymb}
\usepackage{scalefnt}
\usepackage{amsfonts}
\usepackage{color}
\usepackage{xspace}
\usepackage{microtype}
\usepackage[mathscr]{eucal}

\DeclareMathOperator{\Seq}{\subseteq}

\newtheorem{theorem}{Theorem}
\newtheorem{lemma} [theorem] {Lemma}
\newtheorem{proposition}[theorem]{Proposition}
\theoremstyle{theorem}

\newtheorem{corollary}[theorem] {Corollary}
\theoremstyle{definition}
\newtheorem{definition}{Definition}
\theoremstyle{theorem}

\usepackage{algorithm}
\usepackage{algpseudocode}
\def\BState{\State\hskip-\ALG@thistlm}
\makeatother
\newcounter{claimCount}
\newenvironment{claim}{\medskip

    \noindent\refstepcounter{claimCount}\textbf{Claim~\arabic{claimCount}.}}{

    \medskip}
\newenvironment{claimproof}{\noindent\textit{Proof of Claim~\arabic{claimCount}.}}{\hfill\ensuremath{\qedsymbol} \tiny{Claim~\arabic{claimCount}}

    \medskip}
\theoremstyle{definition}
\newtheorem{defn}{Definition}[section]

\theoremstyle{remark}

\theoremstyle{definition}

\theoremstyle{theorem}

\theoremstyle{definition}
\newtheorem*{notation}{Notation}

\begin{document}
\title{Cops and Robbers on Graphs with a Set of Forbidden Induced Subgraphs}	

\author{
Masood Masjoody\\
\texttt{mmasjood@sfu.ca}
\and
Ladislav Stacho\\
\texttt{lstacho@sfu.ca}
}

\maketitle

\begin{abstract}

It is known that the class of all graphs not containing a graph $H$ as an induced subgraph is cop-bounded if and only if $H$ is a forest whose every component is a path. In this study, we characterize all sets $\mathscr{H}$ of graphs with some $k\in \mathbb{N}$ bounding the diameter of members of $\mathscr{H}$ from above, such that $\mathscr{H}$-free graphs, i.e. graphs with no member of $\mathscr{H}$ as an induced subgraph, are cop-bounded. This, in particular, gives a characterization of cop-bounded classes of graphs defined by a finite set of connected graphs as forbidden induced subgraphs. Furthermore, we extend our characterization to the case of cop-bounded classes of graphs defined by a set $\mathscr{H}$ of forbidden graphs such that there is $k\in\mathbb{N}$ bounding the diameter of components of members of $\mathscr{H}$ from above.
\end{abstract}
\section{Introduction}\label{sec:intro}
A {\em game of cops and robbers} is a pursuit game on graphs, or a class of graphs, in which a set of agents, called the {\em cops}, try to get to the same position as another agent, called the {\em robber}. Among several variants of such a game, we solely consider the one introduced in \cite{Aigner}, which is played on finite undirected graphs. Hence, we shall simply refer to this variant as ``the" game of cops and robbers. Let $G$ be a simple undirected graph. Consider a finite set of cops and a robber. The game on $G$ goes as follows. At the beginning of the game (step 1) each cop will be positioned in a vertex of $G$ and then the robber will be positioned in some vertex of $G$. In each of the subsequent steps each agent either moves to a vertex adjacent to its current position or stays still, with the robber taking turn after all of the cops. The cops win in a step $i$ of the game if in that step one of the cops gets to the vertex where the robber is located. The minimum number of cops that are guaranteed to capture the robber on $G$ in a finite number of steps is called the {\em cop number}  of $G$ and denoted $C(G)$. Graph $G$ is said to be {\em $k$-copwin} ($k\in\mathbb{N}$) if $C(G)\le k$. Since the cop number of a graph is equal to the  sum of the cop numbers of its components, whenever the cop number of a graph $G$ is concerned  $G$ is considered to be connected, unless otherwise is stated. A class $\mathscr{G}$ of graphs is called {\em cop-bounded} if there is $k\in\mathbb{N}$ such that $C(G)\le k$ for every $G\in\mathscr{G}$. Among the class of graphs with a bounded cop number one can mention the class of trees, which is cop-bounded by one, and the class of planar graphs, which is cop-bounded by three \cite{Aigner}.

In this paper we study the game of cops and robbers on classes of graphs defined by a set of forbidden induced subgraphs and aim to characterize such classes which are cop-bounded.

\begin{definition}
Let $\mathscr{H}$ be a set of graphs. A graph $G$ is called {\em $\mathscr{H}$-free} if no graph in $\mathscr{H}$ is an induced subgraph of $G$. If $\mathscr{H}$ is a singleton, say $\{H\}$, we shall use $\{H\}$-free and $H$-free interchangeably.
\end{definition}

\noindent Our point of departure is the following theorem which characterizes graphs $H$ such that the class of $H$-free graphs is cop-bounded.
\begin{theorem}[\cite{CDM154}]\label{thm: H-free Cop-bounded} The class of $H$-free graphs has bounded cop number if and only if every connected component of $H$ is a path. 
\end{theorem}

Theorem \ref{thm: H-free Cop-bounded} in
 particular implies that the class of claw-free graphs is not cop-bounded. In this regard, in Section \ref{sec: subclass-claw-free} we present some subclass of claw-free graphs which are cop-bounded and show winning strategies for some constant number of cops on each class. Section \ref{sec: train-chasing} is devoted to presenting a tool (Lemma \ref{lemma: train-chase-robber}) that plays a part in establishing necessary and sufficient conditions on $\mathscr{H}$ for the class of $\mathscr{H}$-free graphs be cop-bounded, in case the diameters of the members of $\mathscr{H}$ have a bound, as presented in Theorem \ref{thm: gen. copbound. diam.}, Section \ref{sec: n-claw-n-net}. In Section \ref{sec: general disconn. cop-bounded} we extend Theorem \ref{thm: gen. copbound. diam.} to the case that the diameters of components of members of $\mathscr{H}$ have a bound (Theorem \ref{thm: gen.gen. copbound. diam.}). 
\begin{notation}
Let $U$ and $W$ be disjoint subsets of the vertex set of a graph $G$. Then we write $U\Leftrightarrow_G W$ (or simply $U\Leftrightarrow W$ if the graph $G$ is understood from the context) to mean that every vertex in $U$ is adjacent to every vertex in $W$. 
\end{notation}

\begin{definition}[Clique Substitution]
Let $G(V,E)$ be a graph and $v\in V(G)$. The {\em clique substitution at} $v$ is the graph obtain from $G$ by replacing $v$ with a clique of size $|N_G(v)|$ and matching vertices in $N_G(v)$ with the vertices of that clique. The {\em clique substitution of} $G$ is the graph obtained from $G$ by performing clique substitutions at all vertices of $G$. We refer to a clique substituted for a vertex of $G$ as a {\em knot} (of $K(G)$).
\end{definition}
\begin{definition} Let $G$ be a graph and $k\in\mathbb{N}$. The operation of introducing $k$ vertices in the interior of every edge of $G$, making edges of $G$ into internally disjoint $k+2$ paths, is called {\em $k$-subdivision} of $G$.
\end{definition}

The aforementioned operations are specially useful in this research, as neither can reduce the cop number of a graph:

\begin{lemma}\cite{CDM154,berarducci1993cop}
The operations of clique substitution and $k$-subdivision (for any $k\in\mathbb{N}$) do not reduce the cop number of any graph.
\end{lemma}

In addition, we shall be using the fact that cubic graphs are not cop-bounded:

\begin{theorem}\cite{andreae1984note}
For every $k\ge 3$ the class of $k$-regular graphs is cop-unbounded.
\end{theorem}


\section{Cops Number of Some Classes of Claw-free Graphs}\label{sec: C-R-c-b-free}\label{sec: subclass-claw-free}
We shall use $H_a$, $H_b$, $H_c$, and $H_d$ to denote claw, bull, net, and antenna, as shown in Figure \ref{fig:b}. 

\begin{figure}[H]
\centering
\begin{minipage}[t]{0.28\textwidth}
\centering
\includegraphics[width=1.35in]{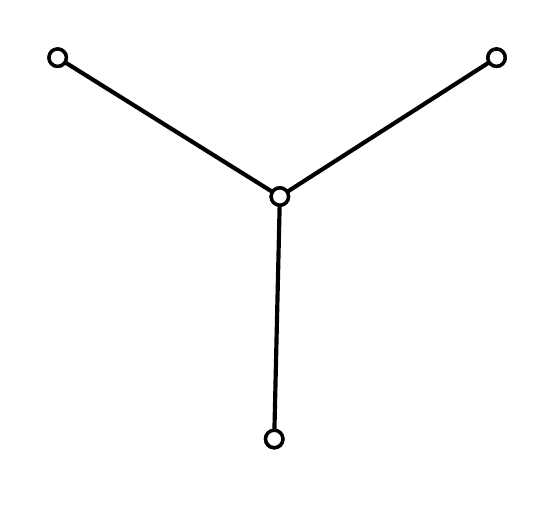}
\caption*{\textbf{(a) }Claw ($H_a$)}
\end{minipage}\hfill\begin{minipage}[t]{0.25\textwidth}
\centering
\includegraphics[width=1.3in]{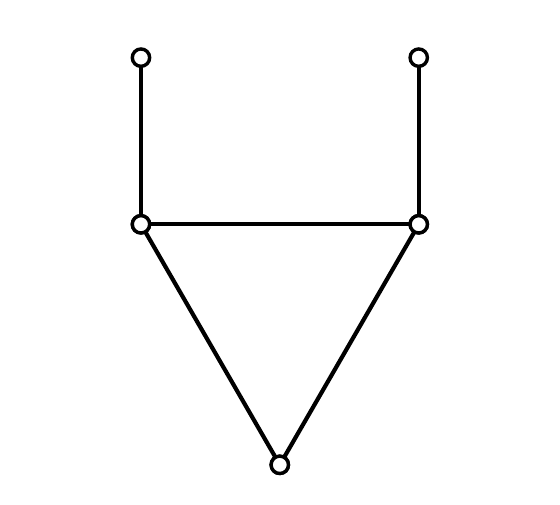}
\caption*{\textbf{(b) }Bull ($H_b$)}
\end{minipage}
\begin{minipage}[t]{0.2\textwidth}
\centering
\includegraphics[width=1.3in]{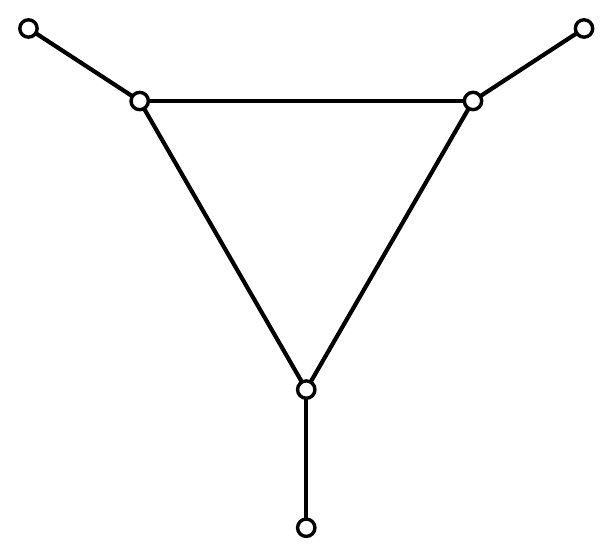}
\caption*{\textbf{(c) }Net ($H_c$)}
\end{minipage}\hfill\begin{minipage}[t]{0.25\textwidth}
\centering
\includegraphics[width=1.3in]{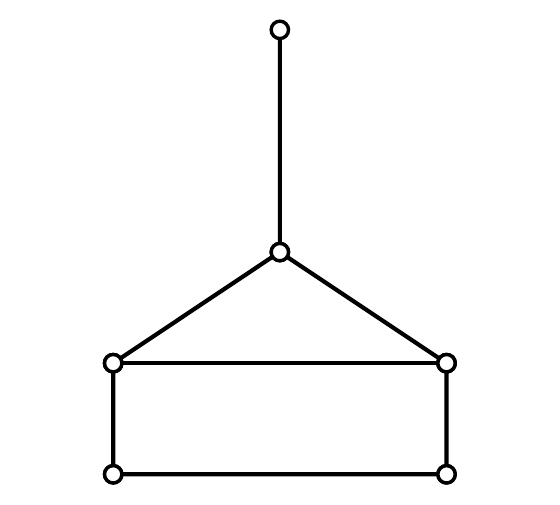}
\caption*{\textbf{(d) }Antenna ($H_d$)}
\end{minipage}\caption{Claw, Bull, Net, and Antenna}
\label{fig:b}
\end{figure}
\begin{defn}[Generalized claw and net]
With $n_1,n_2,n_3\in \mathbb{N}\cup\{0\}$ and $x\in\{a,c\}$, we denote the graph obtained by replacing the degree one vertices in $H_x$ with disjoint paths on $n_1$, $n_2$, and $n_3$ vertices by $H_x(n_1,n_2,n_3)$, calling it a generalized claw or net, according as $x=a$ or $x=c$. For $n\in\mathbb{N}\cup\{0\}$, we denote $H_x(n,n,n)$ simply by $H_x(n)$, calling it the $n$-claw or $n$-net, according as $x=a$ or $x=c$.
\end{defn}
Note that by Theorem \ref{thm: H-free Cop-bounded}, for each $x\in\{a,b,c,d\}$ the class of $H_x$-free graphs is cop-unbounded. We shall consider the game of cops and robbers on the following classes of graphs and for each of them provide an upper bound for the cop number of its members.
\begin{itemize}
\item $\mathbf{Cl}_1:$ the class of connected claw- and bull-free graphs;
\item $\mathbf{Cl}_2:$ the class of connected claw-, net-, and antenna-free graphs;
\item $\mathbf{Cl}_3:$ the class of connected claw- and net-free graphs.
\end{itemize}
In what follows we also denote the class of connected claw-free graphs by $\mathbf{Cl}$.

\begin{lemma}\label{lemma: structure cw-bl-free}
Let $G$ be a connected graph of order $\ge 2$. Pick any two adjacent vertices $u_0$ and $u_1$ of $G$ and let $U$ be the set of neighbor of $u_0$ in $G-u_1$. Let $H$ be the graph obtained from $G$ by removing $U$. Set $N_0=\{u_0\}$, and for each $i\in \mathbb{N}$ let $N_i$ be the $i$th neighborhood of $u_0$ in $H$. In other words, for every $i\in \mathbb{N}\cup \{0\}$ let $N_i=\{v\in V(H): d_H(u_0,v)=i\}$. Then:
\begin{enumerate}
\item If $G\in\mathbf{Cl}$ and $i$ an integer $\ge 2$ any pair of distinct vertices in $N_i$ with a common neighbor in $N_{i-1}$ are adjacent. In particular, $N_2$ is a clique.
\item If $G\in\mathbf{Cl}_1$ then each $N_i$ is a clique. Moreover, $N_i\Leftrightarrow N_{i-1}$ for $i\ge 1$.
\item If $G\in\mathbf{Cl}_2$ then every $N_i$ is a clique and contains a vertex that dominates $N_{i+1}$.
\item If $G\in\mathbf{Cl}_3$ each $N_i$ has independence number $\le 2$. 
\end{enumerate}
\end{lemma}

\begin{figure}[h]
\begin{center}
 \includegraphics[scale=1.35]{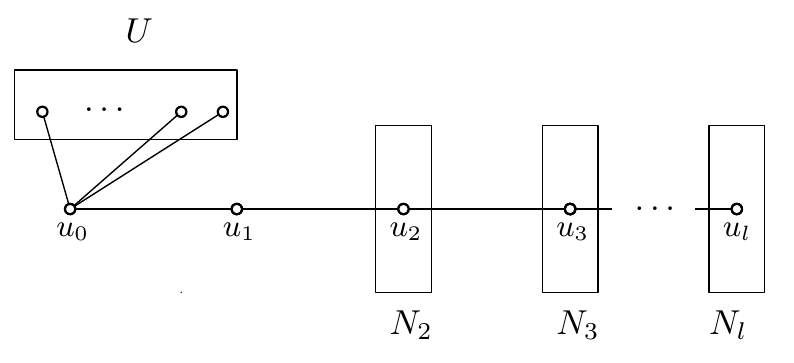}
 \caption{The structure of the graph $H$ obtained by excluding every but one neighbor of a vertex $u_0\in V(G)$. We have $U=N_G(u_0)\setminus\{u_1\}$, $H=G-U$, and $u_0,u_1,\dots,u_l$ is a longest geodesic path in $H$ starting at $u_0$. Each $N_i$ is the bag corresponding to $u_i$, defined as in Lemma \ref{lemma: structure cw-bl-free}} \label{pic: c-b-free-lemma}
\end{center}
\end{figure}
\begin{proof}\textbf{(a) }Let $i\ge 2$ and $u,v$ be distinct vertices in $N_i$ with a common neighbor $w\in N_{i-1}$. By definition, there is a vertex $z\in N_{i-2}$ such that $wz\in E(G)$. As such, one has $zu\not\in E(G)$ and $zv\not\in E(G)$. Therefore, $uv\in E(G)$, for otherwise $G[\{u,v,w,z\}]$ would be a claw, a contradiction. Furthermore, for every pair $x,y$ of distinct vertices in $N_2$ one must have $xy\in E(G)$, for otherwise $G[\{u_0,u_1,x,y\}]$ would be a claw, a contradiction. Hence, $N_2$ is a clique. \noindent\textbf{(b) }According to (a), the claim holds for $i\le2$. Proceeding by induction on $i$, suppose $j\ge 2$ and that the claim holds for all $i\in[1\dcdot j]$. If $N_{j+1}=\varnothing$, the claim vacuously hold for $i=j+1$. Otherwise, for each $k\in [j-2\dcdot j+1]$ choose any $x_{k}\in N_{k}$. Note that by (a) it suffices to show that $x_j x_{j+1}\in E(G)$. To this end, suppose, toward a contradiction, that $x_j x_{j+1}\not\in E(G)$. Pick any $y_j\in N_j$ such that $y_j x_{j+1}\in E(G)$. Then, by the induction hypothesis we have $x_{j-1}y_j,x_j y_j\in E(G)$ and $x_{k-1}x_{k}\in E(G)$ for each $k\in[j-2\dcdot j]$. As such, $G[\{x_{j-2},x_{j-1},x_j,y_j,x_{j+1}\}]$ will be a bull, a contradiction (Figure \ref{pic: c-b-free-lemma-e}).
\begin{figure}[h]
\begin{center}
 \includegraphics[scale=1.25]{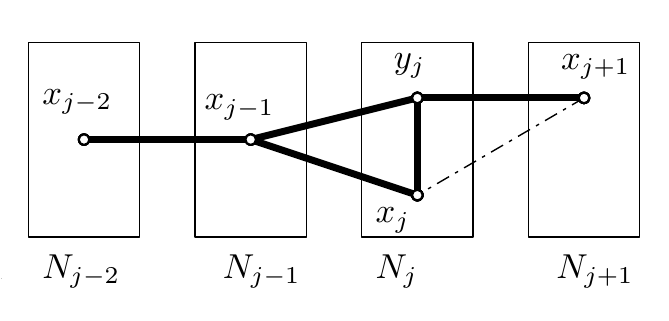}
\caption{Proof of Lemma \ref{lemma: structure cw-bl-free}(b), by induction on $i$: Suppose $j \ge 2$ and (b) holds for every $i\le j$. For each $k$ pick $x_k\in N_k$. If $x_{j+1}x_j \not\in E(G)$, then for each $y_j\in N_j$ with $y_j x_{j+1}\in E(G)$ the graph $G[\{x_{j-2},x_{j-1},x_j,y_j,x_{j+1}\}]$ would be a bull; a contradiction.}\label{pic: c-b-free-lemma-e}
\end{center}
\end{figure}

\noindent\textbf{(c) } According to (a) it suffices to show that every $N_i$ ($i\ge 2$) contains a vertex that dominates $N_{i+1}$. We shall proceed by induction on $i$, noting that this is true for $i\in\{0,1\}$ since $N_0,N_1$ are singletons. Suppose $i\ge 2$ and assume $N_k$ contains a vertex that dominates $N_{k+1}$ whenever $k<i$. Let $z\in N_{i-2}$ and $z'\in N_{i-1}$ such that $N_G(z)\supseteq N_{i-1}$ and $N_G(z')\supseteq N_{i}$. Then, assume, toward a contradiction, that no vertex in $N_i$ is adjacent to all vertices in $N_{i+1}$. The latter implies there exist $x,y\in N_{i}$ and $x',y'\in N_{i+1}$ such that $E[G\{x,x',y,y'\}]=\{xx',yy',xy\}$. But then $G[\{x,x',y,y',z,z'\}]$ would be a net or a monk, according as $xy\not\in E(G)$ or $xy\in E(G)$, a contradiction.
\begin{figure}[h]
\begin{center}
 \includegraphics[scale=1.25]{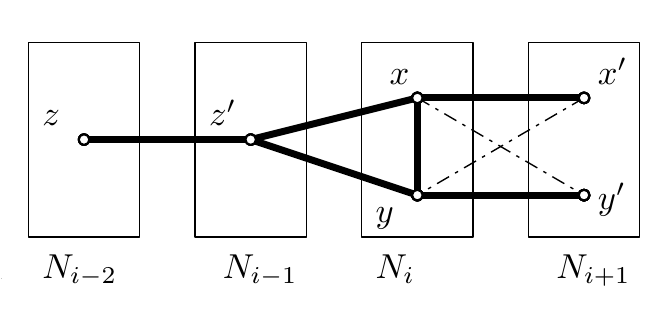}
\caption{Proof of Lemma \ref{lemma: structure cw-bl-free}(c), by induction on $i$. With $i\ge 2$, let $z\in N_{i-2}$ and $z'\in N_{i-1}$ be dominating $N_{i-1}$ and $N_{i}$, respectively. If $(c)$ fails, there will be $x,y\in N_i$ and $x',y'\in N_{i+1}$ such that $E(\{x,y\},\{x',y'\}=\{xx',yy'\}$; thereby $G[\{x,x',y,y',z,z'\}]$ will be a net or an antenna; a contradiction.}\label{pic: c-b-free-lemma-c}
\end{center}
\end{figure}

 \noindent\textbf{(d) }In light of (a) $G[N_i]$ has independence number $\le 2$ whenever $i\le 2$. Proceeding by induction on $i$, suppose $i\ge 3$ and that $G[N_{i-1}]$ has independence number $\le 2$. Toward a contradiction, suppose $G[N_i]$ had independence number $\ge 3$. As such, let $\{x,y,z\}$ be a set of three independent vertices in $N_i$. By (a), no two of $x,y,z$ have a common neighbor in $N_{i-1}$. Therefore, there exist distinct vertices $x',y',z'\in N_{i-1}$ such that $E(\{x,y,z\},\{x',y',z'\})=\{xx',yy',zz'\}$. Moreover, by induction hypothesis, $|E(G[\{x',y',z'\}])|\ge 1$. Assume, without loss of generality, that $x'y'\in E(G)$. If $x',y'$ had a common neighbor $w\in N_{i-2}$, then for every neighbor $w'$ of $z$ in $N_{i-3}$ the graph $G[\{x,x',y,y',w,w'\}]$ would be a net, a contradiction. Hence for any neighbor $x''$ of $x$ in $N_{i-1}$ one has $x''y'\not\in E(G)$; thereby, $G[\{x,x',x'',y\}]$ has to be a claw, a contradiction.
\begin{figure}[h]
\centering
\begin{minipage}[t]{0.45\textwidth}
\centering
\includegraphics[width=2.8in]{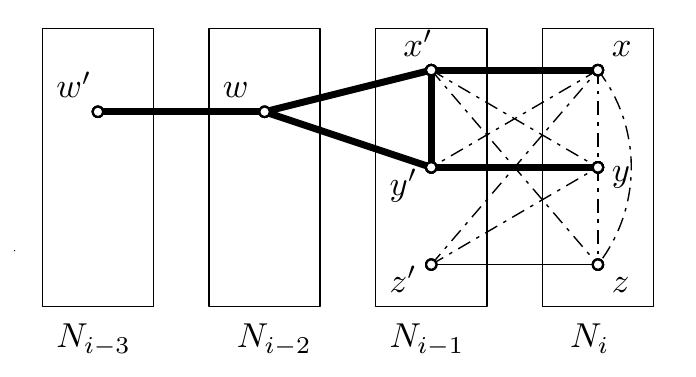}
\caption*{\small{If $x',y'$ have a common neighbor $w\in N_{i-1}$ and if $w'\in N_{i-3}$ is adjacent to $w$ then $G[\{x,x',y,y',w,w'\}]$ is a net.}}
\end{minipage}\hfill\begin{minipage}[t]{0.45\textwidth}
\centering
\includegraphics[width=2.8in]{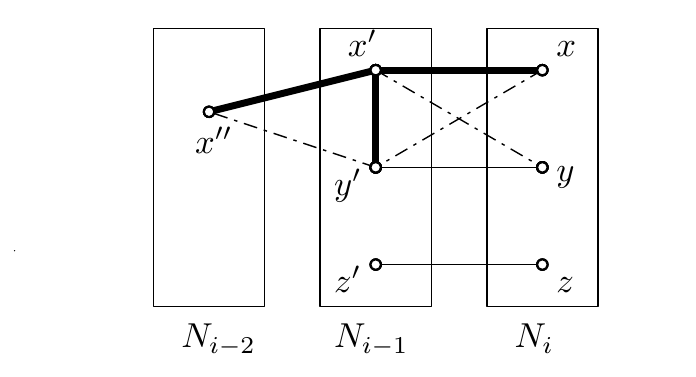}
\caption*{\small{If $x',y'$ have no common neighbor $w\in N_{i-1}$ then for every neighbor $x''$ of $x'$ in $N_{i-2}$ the graph $G[\{x,x',x'',y'\}]$ is a claw.}}
\end{minipage}
\caption{Proof of Lemma \ref{lemma: structure cw-bl-free}(d), by proceeding toward a contradiction. Assume $i\ge 3$, $N_{i-1}$ has independence number $\le 2$, and but $N_{i}$ contains three independent vertices $x,y,z$ with neighbors $x',y',z'$ in $N_{i-1}$ such that $x'y'\in E(G)$.}\label{pic: c-b-free-lemma-d}
\end{figure}
\end{proof}
\begin{theorem}\label{thm: cops-cw-bl-free}\
\begin{enumerate}
\item If $G\in\mathbf{Cl}_1$ then $G$ is two-copwin.
\item If $G\in\mathbf{Cl}_2$ then $G$ is three-copwin.
\item If $G\in\mathbf{Cl}_3$ then $G$ is five-copwin.
\end{enumerate}
\end{theorem}
\begin{proof}In each case, we shall put all the cops in hand initially at the same vertex, say, $u_0$. When the robber takes its first position, say, $r$, consider a geodesic path $P$ in $G$ from $u_0$ to $r$. With $u_1$ being the vertex of $P$ following $u_1$, set $U=N_G(u_0)\setminus \{u_1\}$ and, then, define $H$ and the $N_i$ as in Lemma \ref{lemma: structure cw-bl-free}. Furthermore, let $H'$ be the component of $u_0$ in $H$. For the entire duration of the game keep one cop, say $C_1$, at $u_0$. This, forces the robber to stay in $H'$ and, hence, to the sets $N_i$, according to Lemma \ref{lemma: structure cw-bl-free}. Since $H'$ is finite, there is a unique $k\in \mathbb{N}$ such that $N_k\not=\varnothing$ and $N_{k+1}=\varnothing$.   \textbf{(a) } Let $C_2$ be the other cop in play. By the strategy of moving $C_2$ in $H'$ along any shortest path from $N_0$ to $N_k$, in (at most) $k$ steps $C_2$ either captures the robber or gets to an $N_i$ where the robber is located. In the latter case $C_2$ will be able to capture the robber in its very next move, according to Lemma \ref{lemma: structure cw-bl-free}. \textbf{(b) }Let $C_2, C_3$ be the other cops in play.  According to Lemma \ref{lemma: structure cw-bl-free}(c), there is an induced path $x_0,\dots,x_{k-1}$ in $G$ from $N_0$ to $N_{k-1}$ such that $N[x_j]\supseteq N_{j+1}$ for each $j\le k-1$. Hence, the strategy of moving $C_2$ to $x_{i-1}$ and $C_3$ to $x_i$ in every step $i$ of the game, in no more than $k-1$ steps leads either to the robber's capture ot to having a cop to a vertex neighboring the position of the robber on the cops' turn. \textbf{(c) }Let $C_2, C_3,C_4,C_5$ be the other cops in play. By Lemma \ref{lemma: structure cw-bl-free}(d), for each $j\in[1\dcdot k]$ one can choose $A_j\Seq N_j$ such that $|A_j|\le 2$ and $N[A_j]\Seq N_j$. As such, we initially have $C_2,C_3$ and $C_4,C_5$ saturate $A_0$ and $A_1$, respectively. Note that to avoid being captured, the robber has to stay in $\bigcup_2^k N_j$. We proceed in the following recursive fashion. With two groups of two cops saturating $A_{i-1}$ and $A_i$ for some $i\ge 1$ and the robber being in $\bigcup_{i+1}^k N_j$, in a finite number of steps we will have the cops in $A_{i-1}$ to $A_{i+1}$, while keeping the cops in $A_i$ fixed in their position. The latter guarantees the robber's capture upon leaving $\bigcup_{i+1}^k N_j$. Therefore, having cops in $A_{i+1}$ will push the robber to stay in $\bigcup_{i+1}^k N_j$. Therefore, this presents a strategy for reducing robber's territory; i.e. the cops have a winning strategy. 

\end{proof}

 


\section{Train-chasing the robber}\label{sec: train-chasing}
\begin{proposition}\label{prop: for train-chasing}
Consider an instance of the game of cops and robbers on a graph $G$. Suppose on the cops' turn there are at least two cops $C_1,C_2$ in a vertex $v$ of the graph while the robber is in a vertex $w$. Let $P$ be any $(w,v)$-geodesic path in $G$. Let  $u$ be the second last vertex of $P$ and set $X=N_G(w)\setminus \{u\}$. Moreover, let $H$ be the component of $v$ in $G-X$. Then moving $C_2$ to $u$ and keeping $C_1$ and $C_2$ on $v$ and $u$ for the rest of the game forces the robber to stay in $H$.
\end{proposition}
\begin{proof}
Obvious.
\end{proof}
\begin{defn}
Let $G$ a graph and $U$ be the set of all triples $(u,v,H)$ where $H$ is a connected subgraph of $G$ and $u,v\in V(H)$ with $d_H(u,v)\ge 2$. A {\em chasing function for} $G$ is a function $\theta$ mapping every triple $(u,v,H)\in D$ onto the neighbor of $u$ along a $(u,v)$-geodesic path in $H$. 
\end{defn}
\begin{lemma}\label{lemma: train-chase-robber}
Consider an instance of the game of cops and robber on a graph $G$. Let $\theta$ be a chasing function for $G$. Let $k\in \mathbb{N}$ and suppose on the cops' turn in step one there are $k$ cops $C_1,\dots,C_k$ in a vertex $v_1$ of the graph while the robber is located in a vertex $w_1$. Further suppose the robber has and will use a strategy to survive the the following $k$ moves of $C_1,\dots, C_k$. Denote the following robber's positions with $w_2,\dots,w_{k-1}$. Further, recursively define the $H_i$s ($i\in[1\dcdot k]$) and $v_i$s $i\in[2\dcdot k]$ as follows:
\begin{itemize}
\item $H_1=G$;
\item $v_{i+1}=\theta(v_i,w_i,H_i)$ for each $i\in[1\dcdot k]$;
\item $X_i=N_{H_i}(v_i)\setminus \{v_{i+1}\}$ for each $i\in[1\dcdot k]$;
\item $H_{i+1}:$ the component of $v_1$ in $H_{i}-X_i$ for each $i\in [1\dcdot k]$.
\end{itemize}
Then the following hold:
\begin{enumerate}
\item Every $H_i$ is an induced subgraph of $G$.
\item If $uv\in E(G)\setminus E(H_{k+1})$ such that $u\in V(H_{k+1})$, then $v\in\bigcup_1^{k}X_i$.
\item Vertices $v_1,\dots,v_{k+1}$, in that order, induce a path in $H_k$.
\item The cops can play so than on the cops' turn in step $k$  every $C_i$, $i\in[1\dcdot k]$, is located in vertex $v_i$.
\item Keeping every $C_i$ in $v_i$ for the rest of the game forces the robber to stay in $H_{k+1}$.
\end{enumerate}
\end{lemma}
\begin{proof}
All of the parts follow from Proposition \ref{prop: for train-chasing} and the fact that every $H_{i+1}$ is an induced subgraph of $H_i$ and, hence, of $H_1=G$.(See Figure \ref{pic: train-chase}.)

\begin{center}
 \includegraphics[width=0.8\linewidth]{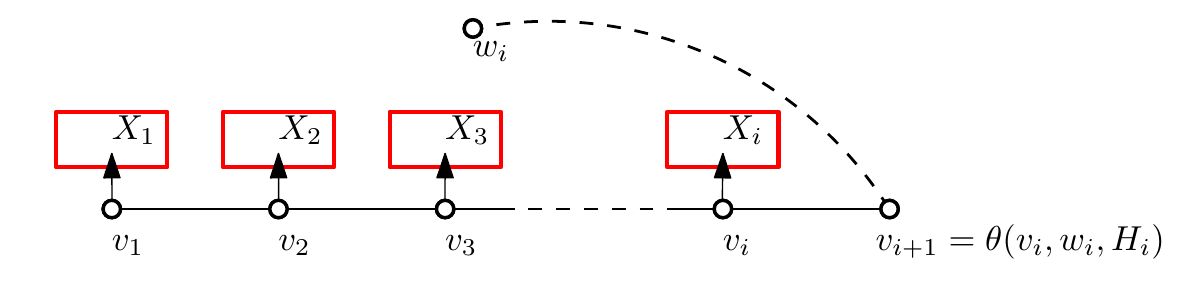}
 \captionof{figure}{Train-chasing the robber according to Lemma \ref{lemma: train-chase-robber}}\label{pic: train-chase}
\end{center}
\end{proof}
\begin{corollary}\label{cor: Pk-free-graphs}
For every integer $k\ge 3$ the class of $P_k$-free graphs is cop-bounded by $k-2$ \cite{CDM154}. Indeed, the cops need no more than $k-1$ steps to capture the robber on a $P_k$-free graph. Moreover, on a $P_k$-free graph ($k\ge 3$) there is a one-active-cop winning strategy for $k-2$ cops. 
\end{corollary}
\begin{proof}
Acoording to Lemma \ref{lemma: train-chase-robber}, starting from every vertex $v_1$ a set of $k-2$ cops can either capture the robber by the end of step $k-2$, or be positioned on $k-2$ vertices that induce a path $P$ in $G$ and restrict the robber to stay in a in induced subgraph $H$ of $G$ that contains $P$. In the latter case, since $G$ is $P_k$-free so will be $H$; thereby, $P$ will be a dominating path for $H$. Hence, the robber will be captured in the very next move of the cops. 
\end{proof}

\section{The Main Result}\label{sec: n-claw-n-net}
In this section we shall prove the following:
\begin{theorem}\label{thm: gen. copbound. diam.}Let $k\in\mathbb{N}$ and $\mathscr{H}$ be a class of graphs such that the diameter of every element in $\mathscr{H}$ is smaller than $k$. Then the class of $\mathscr{H}$-free graphs is cop bounded iff
\begin{enumerate}
\item $\mathscr{H}$ contains a path, or
\item $\mathscr{H}$ contains a generalized claw and a generalized net.
\end{enumerate}
\end{theorem}

\begin{theorem}\label{thm: general claw and net excluded}
If $\mathscr{H}=\{H_a(n),H_c(n)\}$ for some $n\in\mathbb{N}$ then every $\mathscr{H}$-free graph $G$ is $4n$-copwin.
\end{theorem}
\begin{proof}
\setcounter{claimCount}{0}
By Lemma \ref{lemma: train-chase-robber} we may assume that we have the cops initially covering all the vertices of an induced $4n$ path $P^{[4]}:v_1,\dots,v_{4n}$ and that there is path $Q^{[4]}$ in $G$ from robber's position to $v_{4n}$ with a second last vertex $\alpha_4$ such that 
$$\left(V\left(Q^{[4]}\right)\setminus\{\alpha_4,v_{4n}\}\right)\cap N\left(P^{[4]}\right)=\varnothing .$$
Note that as such the robber is forced to stay in the component, say $G_4$, of $v_1$ in $G-(N(V(P^{[4]}))\setminus \{\alpha_4\})$.


We define the present robber's territory by $\mathscr{R}_4:=V(G_4)\setminus N[V(P^{[4]}]$. We put the vertices of $P^{[4]}$ in four disjoint sets of $n$ vertices $V_1,\dots,V_4$, such that $V_i:=\{v_j: j\in[n(i-1)+1\dcdot ni]\}$, and denote the cops covering $V_i$ by $\mathscr{C}_i$. Since $G$ is $\mathscr{H}$-free, for every $w\in\mathscr{R}_4$ we must have
\begin{equation*}
N(w)\cap N(V_2)\Seq N(V_1)\cup N(v_3)\cup N(v_4).
\end{equation*}

\noindent(See Figure \ref{pic: gen-cl-net-free}.) Hence, $\mathscr{C}_2$ can be freed and, hence, moved in the next $3n$ steps to either capture the robber or cover another set $V_5$ of $n$ vertices to be augmented with $V(P^{[4]}$ such that
\begin{enumerate}
\item $V_5\cap N[V(P^{[4]}]=\{\alpha_4\}$, and
\item $V_1,\dots,V_5$ form a $5n$-path $P^{[5]}$ induced in $G_4$ from $v_1$ to, say, $v_{5n}$.
\end{enumerate}
\begin{figure}[H]
\begin{center}
 \includegraphics[width=0.75\linewidth]{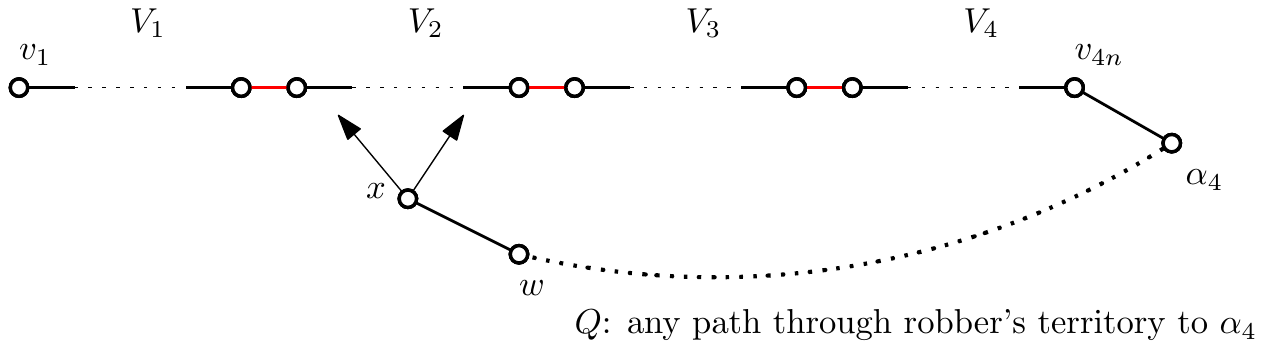}
 \caption{If $wx\in E(G)$ where $w\in\mathscr{R}_4$ and $x\in N(V_2)\setminus N(V_1\cup V_3 \cup V_4)$, for each path $Q$ from $w$ to $\alpha_4$ through $\mathscr{R}_4$ then vertices of $G[(\cup_1^4 V_i)\cup V(Q)\cup \{x\}]$ contains an induced $n$-claw or an induced $n$-net.} \label{pic: gen-cl-net-free}
\end{center}
\end{figure}
We define the subsequent robber's territories, positions of the cops, and paths $P^{[j]}$, $Q^{[j]}$ and vertices $\alpha_j$ ($j> 4$) in the obvious recursive way. Again, as $G$ is $\mathscr{H}$-free, the following holds:
\begin{claim}\label{claim: n-claw-n-net}
Let $w'\in\mathscr{R}_j$ ($j\ge 4$) and suppose we have the sets of cops $\mathscr{C}_1$, $\mathscr{C}_{j-1}$, and $\mathscr{C}_{j}$ covering $V_1$, $V_{j-1}$ and $V_j$. Then for each $i\in[2\dcdot (j-2)]$ we have
\begin{equation*}
N(w')\cap N(V_i)\Seq N(V_1)\cup N(V_{j-1})\cup N(V_j).
\end{equation*}
\end{claim}
\noindent Claim \ref{claim: n-claw-n-net} allows us to get $n$ cops free and then, in light of Lemma \ref{lemma: train-chase-robber}, use them to either capture the robber or augment the current path, say $P^{[q]}$, with a set $V_{q+1}$ of $n$ vertices; thereby shrink the robber's territory. Hence, $4n$ cops have a winning strategy on $G$.


\begin{itemize}
\item robber's territory $\mathscr{R}_4:=V(G_4)\setminus N[V(P^{[4]}]$.
\item $V_i:=\{v_j: j\in[n(i-1)+1\dcdot ni]\}$.
\item $\mathscr{C}_i:$ the set of cops occupying $V_i$.
\end{itemize}
For every $w\in\mathscr{R}_4$ we must have
\begin{equation*}
N(w)\cap N(V_2)\Seq N(V_1)\cup N(V_3)\cup N(V_4),
\end{equation*}
since $G$ is $\mathscr{H}$-free. (See Figure \ref{pic: gen-cl-net-free}.)

Within the next $3n$ steps cops in $\mathscr{C}_2$ can either capture the robber or cover another set $V_5$ of $n$ vertices such that
\begin{itemize}
\item $V_5\cap N[V(P^{[4]}]=\{\alpha_4\}$
\item $V_1,\dots,V_5$ form a $5n$-path $P^{[5]}$ induced in $G_4$ from $v_1$ to, say, $v_{5n}$.
\end{itemize}

\end{proof}

\begin{proof}[Proof of Theorem \ref{thm: gen. copbound. diam.}]\textbf{(Necessity.) }Let $\mathscr{G}_1$ be the class of graphs obtained by $k$-subdividing the cubic graphs, and $\mathscr{G}_2$ be the class obtained by applying clique  substitution to the graphs in $\mathscr{G}_1$. By Theorem \ref{}, $\mathscr{G}_1$ and $\mathscr{G}_2$ are both cop-unbounded. Moreover, the only possible induced subgraphs of graphs in $\mathscr{G}_1$ (resp. $\mathscr{G}_2$)  with diameter $<k$ are paths and generalized claws (reps. paths and generalized nets). \noindent\textbf{(Sufficiency.) }If $\mathscr{H}$ contains a path then copboundedness follows from Corollary \ref{cor: Pk-free-graphs}. Hence, we may assume $\mathscr{H}$ contains an $H_a(n_1,n_2,n_3)$ and an $H_c(m_1,m_2,m_3)$ for some $n_j,m_j\in\mathbb{N}\cup\{0\}$. Let $n$ be the maximum of $n_j,m_j$'s. As such, every $\mathscr{H}$-free graph will be $\{H_a(n),H_c(n)\}$-free graphs. Hence, according to Theorem \ref{thm: general claw and net excluded}, the class of $\mathscr{H}$-free graphs is cop-bounded by $4n$. 

\end{proof}

\section{A Generalization}\label{sec: general disconn. cop-bounded}
In light of Lemma \ref{lemma: train-chase-robber}, one can show the following generalization of Theorem \ref{thm: gen. copbound. diam.}:
\begin{theorem}\label{thm: gen.gen. copbound. diam.}
Let $k\in\mathbb{N}$ and $\mathscr{H}$ be a class of graphs such that the diameter of every component of each element in $\mathscr{H}$ is smaller than $k$. Then the class of $\mathscr{H}$-free graphs is cop bounded iff
\begin{enumerate}
\item $\mathscr{H}$ contains a forest of paths, or
\item $\mathscr{H}$ contains two graphs $F_1$, $F_2$ each having at least one degree three vertex such that every component of $\mathscr{F}_1$ is a path or a generalized claw, and every component of $\mathscr{F}_2$ is a path or a generalized net.
\end{enumerate}
\end{theorem}
\begin{proof}\setcounter{claimCount}{0}
Let $\mathscr{G}_1$ and $\mathscr{G}_2$ be the cop-unbounded classes of graphs as in the proof of Theorem \ref{thm: gen. copbound. diam.}. Then the only induced subgraphs of $\mathscr{G}_1$ with each component of diameter $<k$ are forests with each component either a path or a generalized claw. Similarly, the only induced subgraphs of $\mathscr{G}_2$ with each component of diameter $<k$ are graphs with each component either a path or a generalized net. Hence, if no member of $\mathscr{H}$ is a forest of path, $\mathscr{H}$ must contain a forest  with at least one degree three vertex having every component a path or a  generalized claw. In addition, $\mathscr{H}$ must contain a graph with maximum degree three every component of whose is a path or a generalized net. This establishes the necessity of the conditions. For the reverse direction, likewise in the proof of Theorem it suffices to assume $\mathscr{H}$ is of the form $\{l\cdot P_n+s\cdot H_a(n), l\cdot P_n+t\cdot H_c(n)\}$ where $k,s,t\in\mathbb{N}$. As such, in light of Lemma \ref{lemma: train-chase-robber} with $(n+1)l$ cops and within as many steps of the game the cops can either capture the robber or get positioned to cover all vertices of an $(n+1)k$-path to restrict the robber to a territory which is $\{s\cdot H_a(n),t\cdot H_c(n)\}$-free. Hence, to complete the proof it suffices to show the following:
\begin{claim} Let $n\in\mathbb{N}$ be fixed. For each pair $(s,t)\in\mathbb{N}^2$ denote the class of $\{s\cdot H_a(n), t\cdot H_c(n)\}$-free graphs simply by $\mathscr{G}_{s,t}$. Then $\mathscr{G}_{s,t}$ is cop-bounded by $4(n+1)(s+t-1)$.
\end{claim}
\begin{claimproof}
We shall use induction on $s+t$. The case $s+t=2$, i.e. when $s=t=1$, follows from the Theorem \ref{thm: gen. copbound. diam.}. If $s+t\ge 3$, starting from any vertex $v_1$ we follow the strategy and notation described in the theorem until getting to a path $v_1,\dots,v_{kn}$ with cops on $V_1$, $V_{k-1}$ ,$V_k$ such that there is a vertex in robber's territory adjacent to some $V_i$, $1<i<k-1$, and non-adjacent to $V_1,V_{k-1},V_k$. As such, there will be an $n$-claw or an $n$-net induced on $\dots$, and hence of at most $3(n+1)$ vertices. Then we must have $s\ge 2$ or $t\ge 2$ according as such a graph is an $n$-claw or an $n$-net. Covering the vertices o Without loss of generality, let that be an $n$-claw. Then we must have $s\ge 2$ and keeping some cops on its vertices restricts the robber to a territory in $\mathscr{G}_{s-1,t}$ (with $s\ge 2$) or $\mathscr{G}_{s,t-1}$ (with $t\ge 2$); which will be $4(n+1)(s+t-2)$-copwin, by the induction hypothesis. Hence, every $\mathscr{G}_{s,t}$ is $4(n+1)(s+t-1)$-copbounded, as desired.
\end{claimproof} 

\end{proof}
\section*{Acknowledgments}
We would like to thank Pavol Hell and Bojan Mohar for many useful discussions and for their valuable suggestions throughout this research. We would also like to thank Seyyed Aliasghar Hosseini for useful discussions on some of the methods utilized in \cite{CDM154}. Special thanks to Jozef Hale\u{s} for many useful discussions that led us to Theorem \ref{thm: cops-cw-bl-free}(a).

\bibliographystyle{plain}

\end{document}